\documentclass{amsart}

\usepackage{amsfonts}
\usepackage{amssymb}
\usepackage{enumerate}
\usepackage{graphicx}
\usepackage{amsthm}
\usepackage{amsmath}
\usepackage[english]{babel}

\newcommand{\NN}{\mathbb{N}}

\newcommand{\RR}{\mathbb{R}}

\newtheorem{theorem}{Theorem}[section]
\newtheorem{lemma}[theorem]{Lemma}

\newtheorem{question}[theorem]{Question}

\newtheorem*{t:1}{Theorem \ref{t:1}}

\theoremstyle{definition}

\newtheorem{remark}[theorem]{Remark}
\newtheorem{definition}[theorem]{Definition}

\DeclareMathOperator{\inter}{int}

\DeclareMathOperator{\dist}{dist} \DeclareMathOperator{\Lip}{Lip}
 \DeclareMathOperator{\cl}{cl}
\DeclareMathOperator{\id}{id}

\title{Answer to a question of Kolmogorov}

\author{Rich\'ard Balka}
\address{Alfr\'ed R\'enyi Institute of Mathematics, Hungarian Academy of Sciences,
PO Box 127, 1364 Budapest, Hungary}
\email{balka.richard@renyi.mta.hu}

\thanks{We gratefully acknowledge the support of the
Hungarian Scientific Research Fund grant no.~72655. ME is supported by
the Hungarian Scientific Foundation grants no.~83726 and J\'anos Bolyai
Fellowship. AM is
supported by the EPSRC grant EP/G050678/1.}

\author{M\'arton Elekes}
\address{Alfr\'ed R\'enyi Institute of Mathematics, Hungarian Academy of Sciences,
PO Box 127, 1364 Budapest, Hungary and Institute of Mathematics,
E\"otv\"os Lor\'and University, P\'azm\'any P\'eter s. 1/c, 1117
Budapest, Hungary}
\email{elekes.marton@renyi.mta.hu}

\author{Andr\'as M\'ath\'e}
\address{Mathematics Institute,
University of Warwick, Coventry, CV4 7AL, United Kingdom}
\email{A.Mathe@warwick.ac.uk}



\begin{document}

\begin{abstract} More than 80 years ago Kolmogorov asked the following question. Let $E\subseteq \RR^{2}$
be a measurable set with $\lambda^{2}(E)<\infty$, where $\lambda^2$ denotes the two-dimensional Lebesgue measure.
Does there exist for every $\varepsilon>0$ a contraction $f\colon E\to
\RR^{2}$ such that $\lambda^{2}(f(E))\geq
\lambda^{2}(E)-\varepsilon$ and $f(E)$ is a polygon? We answer this
question in the negative by constructing a bounded, simply connected open
counterexample. Our construction can easily be modified to yield the analogous result in higher dimensions.
\end{abstract}

\maketitle

\section{Introduction}

The following question was posed by M. Laczkovich in \cite{L}. Let $\lambda^{d}$ stand for the $d$-dimensional Lebesgue measure.

\begin{question}[M. Laczkovich] \label{q:L} Let $E\subseteq \RR^{d}$ $(d\geq 2)$ be a
measurable set with $\lambda^{d}(E)>0$. Does there exist a Lipschitz
onto map $f\colon E\to [0,1]^{d}$?
\end{question}

For $d=2$ the positive answer to Question \ref{q:L} follows from a
result of N. X. Uy \cite{Uy}, and D. Preiss also solved this partial
problem by completely different methods. J. Matou\v{s}ek \cite{M} proved the
following stronger, `absolute constant' version based on a
well-known combinatorial lemma due to Erd\H{o}s and Szekeres.
(For the definition of $1$-Lipschitz map see the Preliminary section.)

\begin{theorem}[J. Matou\v{s}ek] \label{t:M} There exists a constant $c>0$ such that
for any measurable set $E\subseteq \RR^{2}$ with $\lambda^{2}(E)=1$
there exists a 1-Lipschitz onto map $f\colon E\to [0,c]^{2}$.
\end{theorem}

Question \ref{q:L} is still open for dimensions $d>2$. Theorem
\ref{t:M} states that we can contract every set of the plane with
positive measure onto a square such that it `does not lose too much
from its measure'. Can we do this so that the loss of the measure
is arbitrarily small? It is easy to see that this is not possible if we
require the range to be a square, but how about polygons? Note that by polygons we mean a
wider class of objects than its standard definition does:

\begin{definition}  We say that $P\subseteq \RR^2$ is a
\emph{polygon} if $\partial P$ can be covered by finitely many line segments.
\end{definition}

The next question is due to A. N. Kolmogorov, it was quoted by
P. Alexandroff in a letter written to F. Hausdorff, see \cite{A}
and \cite{E}.

\begin{question}[A. N. Kolmogorov] \label{q} Let $E\subseteq \RR^{2}$ be a measurable set with
$\lambda^{2}(E)<\infty$, and let $\varepsilon>0$. Does there exist a
contraction $f\colon E\to \RR^{2}$ such that $\lambda^{2}(f(E))\geq
\lambda^{2}(E)-\varepsilon$ and $f(E)$ is a polygon?
\end{question}

The main goal of the paper is to answer Question \ref{q} in the negative.

\begin{theorem} \label{thm} There exist a bounded, simply connected open set $U\subseteq
\RR^{2}$ and $\varepsilon>0$ such that if
$f\colon U\to \RR^2$ is a contraction with
$\lambda^{2}(f(U))\geq \lambda^{2}(U)-\varepsilon$ then $f(U)$
is not a polygon.
\end{theorem}

In contrast to Question \ref{q:L} the higher dimensional versions
of Question \ref{q} are not more difficult than the original one.
The analogue of Theorem \ref{thm} can be proved similarly for every dimension $d>2$
with the straightforward modifications.

The structure of the paper is as follows. In Section \ref{s:prel} we recall some notation and definitions which we use in this paper.
In Section \ref{s:proof} we prove Theorem \ref{thm}. Finally, in Section \ref{s:open} we collect the open problems.

\section{Preliminaries} \label{s:prel}

Let $B(x,r)$ stand for the closed ball
of radius $r$ centered at $x$. For a set $A\subseteq
\RR^2$ we denote by $\inter A$, $\cl A$ and $\partial A$ the
interior, closure and boundary of $A$, respectively.
A function $f\colon A\to
\RR^{2}$ is said to be \emph{Lipschitz} if there exists a constant $c \in \RR$
such that $|f(x)-f(y)|\leq c |x-y|$ for all $x,y\in A$. The smallest
such constant $c$ is called the Lipschitz constant of $f$ and
denoted by $\Lip(f)$. If $\Lip(f)\leq 1$ then $f$ is a
\emph{1-Lipschitz map}, if $\Lip(f)<1$ then $f$ is a
\emph{contraction}. If $A,B \subseteq \RR^2$ then let
$\dist(A,B)=\inf \{|x-y|: x \in A, y \in B\}$.

For the sake of simplicity, let $\lambda= \lambda^{2}$ stand for the two-dimensional Lebesgue measure.

\section{The proof} \label{s:proof}

First we need the following lemma.

\begin{lemma} \label{l:2} Assume $U\subseteq \RR^{2}$ is a bounded, connected open set and
$f\colon U\to \RR^{2}$ is a 1-Lipschitz map such that $\lambda(f(U))=\lambda(U)$. Then
$f$ is an isometry.
\end{lemma}

\begin{proof}
Recall that a map $g\colon A\to\RR^2$ is an isometry if it preserves distances and that this implies that $g$ is a linear map restricted to the set $A\subseteq \RR^2$.

First we prove that $f$ is locally an isometry at every point of $U$. Let $z\in U$ be arbitrary. Then there exists $r>0$ such that $B(z,2r)\subseteq U$. We prove that $f$ is isometry on $B(z,r)$; that is, $|f(x)-f(y)|=|x-y|$ for all $x,y\in B(z,r)$. As $f$ is $1$-Lipschitz, we have $|f(x)-f(y)|\leq |x-y|$ for all $x,y\in U$. Assume to the contrary that there are $x,y\in B(z,r)$ such that $|f(x)-f(y)|<|x-y|$. Let
$$C_{x,y}=B(x, |x-y|/2) \cup B(y, |x-y|/2)$$
be the union of two closed discs of center $x$ and $y$ and radius equal to $|x-y|/2$.
Clearly, $C_{x,y}\subseteq B(z,2r)\subseteq U$. Since $f$ is $1$-Lipschitz, $f(C_{x,y})$ is contained in the union of two discs
$B(f(x), |x-y|/2) \cup B(f(y), |x-y|/2)$. Since $|f(x)-f(y)|<|x-y|$, the area of this union is smaller than the area of $C_{x,y}$, thus $\lambda(f(C_{x,y}))<\lambda(C_{x,y})$.
Applying that $\lambda$ is subadditive and $f$ is $1$-Lipchitz we obtain
$$\lambda\left(f(U)\right)\leq \lambda\left(f(C_{x,y})\right)+\lambda\left(f(U\setminus C_{x,y})\right)<\lambda(C_{x,y})+\lambda\left(U\setminus C_{x,y}\right)=\lambda(U),$$
which is a contradiction.

Since $f$ is locally an isometry at every point of the open and connected set $U$,
all local isometries are restrictions of the same linear map $\RR^2\to\RR^2$. Hence $f$ is an isometry.
\end{proof}

Now we are ready to prove Theorem \ref{thm}.

\begin{proof}[Proof of Theorem \ref{thm}]
Let $B$ be the closed unit ball centered at the origin and
let $C\subseteq [0,1]$ be a nowhere dense compact set with positive one-dimensional Lebesgue measure.
Set $U=\inter(B)\setminus (C\times [0,1])$.
Clearly, $U$ is open and path-connected. It is easy to see that
every simple closed curve can be shrunk to a
point continuously in $U$, so $U$ is simply connected. Clearly, $\cl U =B$ and
$\lambda(U)<\lambda(B)$.

It is enough to prove that there is an $\varepsilon>0$ such that if
$f\colon U\to \RR^2$ is a contraction with $\lambda(f(U))\geq \lambda(U)-\varepsilon$ then
$\lambda\left(\partial\left( f(U)\right)\right)>0$. Assume to the
contrary that for all $n\in \NN^{+}$ there are contractions
$g_n\colon U \to \RR^{2}$ such that
$\lambda(g_{n}(U))\geq \lambda(U)-1/n$ and $\lambda \left(\partial(g_n(U))\right)=0$.
Clearly, we may assume that $\bigcup_{n=1}^{\infty}
g_{n}(U)$ is bounded. Let $\{z_i: i\in \NN\}$ be a dense set in
$U$. By Cantor's diagonal argument we can choose a strictly increasing
subsequence of
the positive integers $\langle n_k \rangle$ such that
for every $i\in \NN$ the limit $\lim_{k\to \infty} g_{n_k}(z_i)$
exists. Since the maps $g_{n_k}$ are contractions, the sequence of functions $\langle g_{n_k} \rangle$ is uniformly
convergent on $U$. Therefore we may assume that $g_{n}$ converges uniformly to $g$ for a map $g\colon U\to \RR^2$.
The uniform convergence implies that $g$ is 1-Lipschitz.

First we prove that $\lambda(g(U))=\lambda(U)$. Since $g$ is $1$-Lipschitz, $\lambda(g(U))\leq \lambda(U)$,
so it is enough to prove the opposite direction. As a continuous image of an
open set, $g(U)$ is $\sigma$-compact, so measurable. Let $\delta>0$ be arbitrary. The regularity of the
Lebesgue measure implies that there is an open set $V$ such that $g(U)\subseteq V$ and $\lambda(V)<\lambda(g(U))+\delta$. Similarly, there exists a
compact set $K$ such that $K\subseteq U$ and $\lambda(U\setminus K)<\delta$. Since the maps $g_n$ are contractions, we obtain for all $n\in \NN^+$
\begin{equation} \label{eq:gnU} \lambda(g_n(U))-\lambda(g_n(K))\leq \lambda(g_n(U\setminus K))\leq \lambda(U\setminus K)<\delta.
\end{equation}
The uniform convergence
$g_n\to g$ yields that there is an integer $L$ such
that for all $n>L$ we have $g_{n}(K)\subseteq V$. Therefore \eqref{eq:gnU} and the definition of
the maps $g_{n}$ imply that for all $n>L$ we have
$$\lambda(g(U))+\delta > \lambda(V) \geq
\lambda(g_{n}(K)) > \lambda(g_n(U))-\delta \geq \lambda(U)-1/n-\delta.$$
As $\delta >0$ is arbitrary, we obtain $\lambda (g(U))\geq  \lambda(U)$, so $\lambda(g(U))=\lambda(U)$.
Then Lemma \ref{l:2} implies that $g$ is an isometry. We may assume that $g$ is the identity, that is, $g=\id_{U}$.

Since $\cl U=B$, one can extend the maps $g_n$ to contractions $\widehat{g}_{n}\colon B\to \RR^{2}$.
Clearly, $\widehat{g}_n\to \id_{B}$ uniformly on $B$.
Let $D\subseteq B$ be a closed ball centered at the origin
such that $\lambda(U)<\lambda(D)<\lambda(B)$. There exists
$M\in \NN$ such that for all $n>M$ we have
\begin{equation} \label{e:1} \max_{x\in B}\left|\widehat{g}_n(x)-x\right|<\dist(D,\partial B).
\end{equation}
We prove that for all $n>M$,
\begin{equation} \label{e:2} D\subseteq \cl(g_{n}(U)).
\end{equation}
Let us fix $n>M$. As $g_{n}(U)$ is dense in $\widehat{g}_{n}(B)$, we obtain
$\cl(g_{n}(U))=\widehat{g}_{n}(B)$. Thus we need to prove $D \subseteq \widehat{g}_{n}
(B)$ for \eqref{e:2}. Assume to the contrary that there is an $x_0\in D$ such
that $x_0 \notin \widehat{g}_{n} (B)$. Set $r=\dist(D,\partial B)$. Then $B(x_0,r)\subseteq B$, so we can define the map
$\phi \colon B(x_0,r) \to \RR^2$ by $\phi(x)=-\widehat{g}_{n}(x)+x+x_0$. Inequality \eqref{e:1} implies $|\phi(x)-x_0|<r$, so
$\phi(B(x_0,r))\subseteq B(x_0,r)$. Since $x_0 \notin \widehat{g}_{n} (B)$, we obtain that
$\phi(x)\neq x$ for all $x\in B(x_0,r)$. Hence $\phi$ is a continuous self-map of the ball $B(x_0,r)$
without any fixed points, which contradicts the Brouwer Fixed Point Theorem \cite[Proposition 4.4.]{Fu}. Thus \eqref{e:2} holds.

As the maps $g_n$ are contractions, we have $\lambda(g_{n}(U))\leq
\lambda(U)$. Therefore $\lambda(U)<\lambda(D)$ and \eqref{e:2}
imply, for all $n>M$, that
\begin{align*} \lambda(\partial g_n(U))&\geq \lambda(\cl(g_{n}(U))\setminus g_{n}(U)) \\
&\geq \lambda(\cl(g_{n}(U)))-\lambda(g_{n}(U)) \\
&\geq \lambda(D)-\lambda(U)>0.
\end{align*}
Thus $\lambda(\partial g_n(U))>0$, which contradicts the definition of $g_n$. The proof is complete.
\end{proof}

\begin{remark} It is easy to see that for all Lebesgue null sets $N\subseteq \RR^2$ the sets $U\Delta N$ are also counterexamples to Question \ref{q}.
On the other hand, one can show that for all $\varepsilon>0$ there exist a contraction $f\colon U\to \RR^2$ and a Lebesgue null set
$N\subseteq \RR^2$ such that $\lambda^{2}(f(U))\geq \lambda^{2}(U)-\varepsilon$ and $f(U)\Delta N$ is a polygon.
Thus $U$ will not be a counterexample to Question \ref{q:4} below.
\end{remark}

\section{Open Questions} \label{s:open}

The most important question is the following.

\begin{question} \label{q:2} Let $K\subseteq
\RR^{2}$ be a compact set, and let $\varepsilon>0$. Does there exist a
contraction $f\colon K\to \RR^{2}$ such that $\lambda^{2}(f(K))\geq
\lambda^{2}(K)-\varepsilon$ and $f(K)$ is a polygon?
\end{question}

In order to answer Question \ref{q:2} we consider the next
question.

\begin{question} \label{q:3} Let $C\subseteq
\RR^{2}$ be a compact set with $\lambda^{2}(C)=0$, and
let $\varepsilon>0$. Does there exist a contraction $f\colon C\to
\RR^{2}$ such that $|f(x)-x|\leq \varepsilon$ for all $x\in C$ and
$f(C)$ can be covered by finitely many line segments?
\end{question}

If the compact set $C$ is a counterexample to Question
\ref{q:3} with $\varepsilon>0$, then consider $K=C\cup R$, where $R$ is a closed ring
such that the bounded component of its complement contains $C$.
Then $K$ is a counterexample to Question \ref{q:2}, the sketch of the proof is the following. Assume to the
contrary that there are contractions $f_n\colon K\to \RR^{2}$ $(n\in
\NN^{+})$ such that $\lambda(f_{n}(K))\geq \lambda(K)-1/n$ and
$f_{n}(K)$ is a polygon, that is, $\partial f_{n}(K)$ can be covered by finitely many line segments.
Similarly as in the proof of Theorem \ref{thm}, one can show that $f_{n}$ converges uniformly to an isometry, $f$. We
may assume that $f=\id_{K}$. Let us fix $n\in \NN^{+}$ such that $|f_n(x)-x|\leq \varepsilon$ for all $x\in C$ and $f_n(C)\cap f_{n}(R)=\emptyset$.
As $f_n$ is a contraction and $C$ has zero Lebesgue measure, $f_n(C)$ also has zero measure, so $\dist(f_n(C),f_n(R))>0$ implies $f_n(C)\subseteq \partial f_{n}(K)$.
Therefore $f_n(C)$ can be covered by finitely many line segments, which contradicts the choice of $C$ and $\varepsilon$.

\begin{remark} We do not know whether the Sierpi\'nski carpet is a counterexample to Question \ref{q:3}.
\end{remark}

Finally, our last question is the following.

\begin{question} \label{q:4} Let $E\subseteq \RR^{2}$ be a measurable set with
$\lambda^{2}(E)<\infty$, and let $\varepsilon>0$. Do there exist a
contraction $f\colon E\to \RR^{2}$ and a Lebesgue null set $N\subseteq \RR^2$ such that $\lambda^{2}(f(E))\geq
\lambda^{2}(E)-\varepsilon$ and $f(E)\Delta N$ is a polygon? Is this true at least for compact sets?
\end{question}

\subsection*{Acknowledgements}
The authors are indebted to M.~Laczkovich for some valuable remarks and helpful suggestions.


\begin{thebibliography}{99}

\bibitem{A} P.~Alexandroff, letter to Felix Hausdorff, 27 December 1932.

\bibitem{E} J. Elstrodt, Alte Briefen -- aktuelle Fragen,
\textit{Mitt. Dtsch. Math.-Ver.}, \textbf{18} (2010), no. 3,
183--187.

\bibitem{Fu} W. Fulton, \textit{Algebraic Topology. A First
Course.} Springer-Verlag, 1995.

\bibitem{L} M. Laczkovich, Paradoxical decompositions using
Lipschitz functions, \textit{Real Anal. Exchange}, \textbf{17} (1991-1992),
439--443.

\bibitem{M} J. Matou\v{s}ek, On Lipschitz mapping onto a
square,
\textit{The mathematics of Paul Erd\H{o}s, II,}
({Algorithms and Combinatorics} {14})
303--309,
Springer, Berlin, 1997.

\bibitem{Uy} N. X. Uy, Removable sets of analytic functions
satisfying a Lipschitz condition, \textit{Ark. Mat.} \textbf{17} (1979), no.
1, 19--27.






\end{thebibliography}
\end{document}